\documentclass[a4paper,12pt]{amsart}
\usepackage{amssymb}
\usepackage{ifthen}
\usepackage{graphicx}
\usepackage{float}
\usepackage{caption}
\usepackage{subcaption}
\usepackage{cite}
\usepackage{amsfonts}
\usepackage{amscd}
\usepackage{amsxtra}
\usepackage{color}

\newtheorem{thm}{Theorem}
\newtheorem{cor}{Corollary}
\newtheorem{lem}{Lemma}

\newtheorem{conj}{Conjecture}
\newtheorem{prob}{Problem}
\newtheorem{remark}{Remark}

\theoremstyle{definition}
\newtheorem{defn}{Definition}[section]
\newtheorem{example}{Example}

\newenvironment{pf}[1][]{%
 \vskip 1mm
 \noindent
 \ifthenelse{\equal{#1}{}}%
  {{\slshape Proof. }}%
  {{\slshape #1.} }%
 }%
{\qed\bigskip}

\newcounter{alphabet}
\newcounter{tmp}

\makeatletter
\newcommand{\Ref}[1]{\@ifundefined{r@#1}{}{\setcounter{tmp}{\ref{#1}}\Alph{tmp}}}
\makeatother

\newcommand{\IC}{{\mathbb C}}
\newcommand{\ID}{{\mathbb D}}

\def\be{\begin{equation}}
\def\ee{\end{equation}}

\newcommand{\bee}{\begin{enumerate}}
\newcommand{\eee}{\end{enumerate}}

\newcommand{\blem}{\begin{lem}}
\newcommand{\elem}{\end{lem}}
\newcommand{\bthm}{\begin{thm}}
\newcommand{\ethm}{\end{thm}}
\newcommand{\bcor}{\begin{cor}}
\newcommand{\ecor}{\end{cor}}
\newcommand{\beg}{\begin{example}}
\newcommand{\eeg}{\end{example}}
\newcommand{\begs}{\begin{examples}}
\newcommand{\eegs}{\end{examples}}
\newcommand{\bdefe}{\begin{defn}}
\newcommand{\edefe}{\end{defn}}
\newcommand{\bprob}{\begin{prob}}
\newcommand{\eprob}{\end{prob}}
\newcommand{\bques}{\begin{ques}}
\newcommand{\eques}{\end{ques}}
\newcommand{\bei}{\begin{itemize}}
\newcommand{\eei}{\end{itemize}}
\newcommand{\bcon}{\begin{conj}}
\newcommand{\econ}{\end{conj}}
\newcommand{\bcons}{\begin{conjs}}
\newcommand{\econs}{\end{conjs}}
\newcommand{\bprop}{\begin{propo}}
\newcommand{\eprop}{\end{propo}}
\newcommand{\br}{\begin{rem}}
\newcommand{\er}{\end{rem}}
\newcommand{\brs}{\begin{rems}}
\newcommand{\ers}{\end{rems}}
\newcommand{\bo}{\begin{obser}}
\newcommand{\eo}{\end{obser}}
\newcommand{\bos}{\begin{obsers}}
\newcommand{\eos}{\end{obsers}}
\newcommand{\bpf}{\begin{pf}}
\newcommand{\epf}{\end{pf}}
\newcommand{\ba}{\begin{array}}
\newcommand{\ea}{\end{array}}
\newcommand{\beq}{\begin{eqnarray}}
\newcommand{\beqq}{\begin{eqnarray*}}
\newcommand{\eeq}{\end{eqnarray}}
\newcommand{\eeqq}{\end{eqnarray*}}

\newcounter{minutes}
\divide\time by 60
\newcounter{hours}
\multiply\time by 60 \addtocounter{minutes}{-\time}

\begin{document}
\bibliographystyle{amsplain}
\title[Bohr inequality for odd analytic functions]{Bohr inequality for odd analytic functions}

\thanks{
File:~\jobname .tex,
          printed: \number\day-\number\month-\number\year,
          \thehours.\ifnum\theminutes<10{0}\fi\theminutes}


\author{Ilgiz R Kayumov, and Saminathan Ponnusamy }

\address{I. R Kayumov, Kazan Federal University, Kremlevskaya 18, 420 008 Kazan, Russia
}
\email{ikayumov@kpfu.ru }

\address{S. Ponnusamy,
Indian Statistical Institute (ISI), Chennai Centre, SETS (Society
for Electronic Transactions and Security), MGR Knowledge City, CIT
Campus, Taramani, Chennai 600 113, India.
}
\email{samy@isichennai.res.in, samy@iitm.ac.in}

\subjclass[2000]{Primary 30A10, 30H05, 30C35; Secondary 30C45}
\keywords{Analytic functions, $p$-symmetric functions, Bohr's inequality, Schwarz lemma, subordination and odd univalent functions}

\begin{abstract}
We determine the Bohr radius for the class of odd functions $f$ satisfying $|f(z)|\le 1$ for all $|z|<1$, settling the recent conjecture of
 Ali, Barnard and Solynin \cite{AliBarSoly}.
In fact,  we solve this problem in a more general setting. Then we discuss
Bohr's radius for the class of analytic functions $g$, when $g$ is subordinate to a member of the class of odd univalent functions.
\end{abstract}

\thanks{
}

\maketitle
\pagestyle{myheadings}
\markboth{I. R. Kayumov and S. Ponnusamy}{Bohr inequality for odd analytic functions}

\section{Preliminaries and Main Results}
Let $\mathcal A$ denote the space of all functions analytic in the unit disk $\ID :=\{z\in\IC:\, |z|<1\}$ equipped with the topology of
uniform convergence on compact subsets of $\ID$. Then the classical Bohr's inequality \cite{Bohr-14}
states that if a power series $f(z)=\sum_{n=0}^{\infty} a_nz^n$ belongs to $\mathcal A$ and $|f(z)|<1$ for all $z\in \ID$, then
$M_f(r):=\sum_{n=0}^{\infty}|a_n|r^n \leq 1$ for all
$|z|=r\leq 1/3$ and the constant $1/3$ cannot be improved. The constant $r_0=1/3$ is known as Bohr's radius.
Bohr actually obtained the inequality for $r\leq 1/6$, but subsequently later, Wiener, Riesz and Schur, independently established
the sharp inequality for $|z|\leq 1/3$. For a detailed account of the development, we refer to the recent survey article
on this topic \cite{AAPon1} and the references therein.

A variety of results on Bohr's theorem in several complex variables appeared recently. See \cite{Pop} and the references there.
For example, Boas and Khavinson \cite{BoasKhavin-97-4} obtained some multidimensional generalizations of Bohr's theorem
and Aizenberg \cite{Aizen-00-1,Aizen-05-3} extended it for further studies in this topic. Using Bohr's inequality,
Dixon \cite{Dixon-95-7} constructed an example of a Banach algebra that satisfies von Neumann's
inequality but not isomorphic to the algebra of bounded operators on a Hilbert
space. There has been considerable interest after the appearance of the work of Dixon.
Paulsen and Singh extended Bohr's inequality to Banach algebras in \cite{PaulSingh-04-11}.
In \cite{BenDahKha},  the Bohr phenomenon for functions in Hardy spaces is discussed.
In \cite{BaluCQ-2006}, Balasubramanian {\it et al.} extended the Bohr inequality to the setting of Dirichlet series.
For certain other results on the Bohr phenomenon, we refer to \cite{Abu,Abu2, Abu4, AAD2,AizenTark-01-2}. 

The present investigation is motivated by the following problem of Ali, Barnard and Solynin \cite{AliBarSoly}.

\bprob {\rm(\cite{AliBarSoly})\label{KP2-prob1}}
 Find the Bohr radius for the class of odd functions $f$ satisfying $|f(z)|\le 1$ for all $z\in
\mathbb{D}$.
\eprob

In \cite[Lemma 2.2]{AliBarSoly}, it was shown that   $M_f(r)\leq 1$ for all
$|z|=r\leq r_*$,  where $r_*$ is a solution of the equation
$$ 5r^4+4r^3-2r^2-4r+1=0,
$$
which is unique in the interval $1/\sqrt{3}<r<1$. The value of $r_*$  can be calculated in terms of radicals and it is equal to $0.7313\ldots$.

Moreover, in \cite{AliBarSoly}, an example was also given to conclude that the
Bohr radius for the class of odd functions satisfies the
inequalities $r_*\le r\le r^*\approx 0.789991 $, where
\begin{equation} \label{TrueNumber}
 r^*= \frac{1}{4}\sqrt{\frac{B-2}{6}}+\frac{1}{2}
\sqrt{3\sqrt{\frac{6}{B-2}}-\frac{B}{24}-\frac{1}{6}},
\end{equation}
with
$$B=(3601-192\sqrt{327})^{\frac{1}{3}}+(3601+192\sqrt{327})^{\frac{1}{3}}.
$$
One of the aims of this article is to solve this conjecture in a more general form.
Namely, we are going to solve analogous problem for $p$--symmetric functions of the form
$f(z)=z\sum_{k=0}^{\infty} a_{pk+1}z^{pk}$.
We should remark that $p$--symmetric property in case $p>1$
brings serious difficulties because if we use sharp inequalities
$|a_n| \leq 1-|a_0|^2$ $(n\geq 1)$
simultaneously (as in classical case) then we will not obtain the
sharp result due to the reason that in the extremal case  $|a_0|<1$.
Also it is important that in the classical case there is no extremal
function while in our case there is. We now state our main results and their corollaries. The proofs
of the main results will be presented in Section \ref{KP2-sec2}.


\begin{thm}\label{KP2-th3}
Let $p \in \mathbb{N}$,  $f(z)$ be analytic and $p$--symmetric in $\ID$ such that $|f(z)| \leq 1$ in $\ID$. Then
$$M_f(r) \leq 1 \mbox{ for } r \leq r_p,
$$
where $r_p$ is the maximal positive root of the equation
$$-6 r^{p-1} + r^{2(p-1)} + 8 r^{2p} +1= 0
$$
in $(0,1)$. The extremal function has the form $z(z^p-a)/(1-az^p)$, where
$$a=\left (1-\frac{\sqrt{1-{r_p}^{2p}}}{\sqrt{2}}\right )\frac{1}{{r_p}^p }.
$$
\end{thm}

The result for $p=1$ is well known with $r_1=1/\sqrt{2}$.
%

The case $p=2$ has a special interest since it provides a solution to Problem \ref{KP2-prob1}.

\begin{cor}\label{KP2-cor1}
If $f(z)$ is odd analytic in $\ID$ and $|f(z)| \leq 1$ in $\ID$, then
$$M_f(r) \leq 1 \mbox{ for } r \leq r_2=0.789991\ldots ,
$$
where $r_2=r^*$  coincides with \eqref{TrueNumber}.
The extremal function has the form $z(z^2-a)/(1-az^2)$.
\end{cor}


To state our next result, we  need to introduce the notion of subordination. Let $f,g\in {\mathcal A}$. Then $g$ is {\it subordinate} to $f$,
written $g\prec f$ or $g(z)\prec f(z)$, if there exists a
$w\in {\mathcal A}$ satisfying $w(0)=0$, $|w(z)|<1$ and
$g(z)=f(w(z))$ for $z\in \ID$.  In the case when $f$ is univalent in $\mathbb{D}$,  $g\prec f$ if and only
if $g(0)=f(0)$ and $g(\ID )\subset f(\ID )$ (see \cite[Chapter 2]{AvWir-09} and \cite[ p.~190 and p.~253]{DurenUniv-83-8}).
By the Schwarz lemma, it follows that
$$|g'(0)|= |f'(w(0))w'(0)| \leq |f'(0)|.
$$
For  important discussions on the Schwarz lemma and its various consequences, we refer to \cite{AvWir-09}.

Now for a given $f$, let $S(f)=\{g:\, g\prec f\}$. In \cite[Theorem 1]{Abu}, it was shown that if $f,g\in {\mathcal A}$
such that $f$ is univalent in $\ID$ and $g\in S(f)$, then the inequality $M_g(r) \leq 1 $
holds with $r_f=3-2\sqrt{2}\approx 0.17157$. The sharpness of $r_f$ is shown by the Koebe function $f(z)=z/(1-z)^2.$
Our next result concerns Bohr's radius for the space of subordinations when the subordinating function is odd and univalent in $\ID$.
In this case, the Bohr radius is much larger and the proof in this case is completely different.
Unlike the earlier case \cite[Theorem 1]{Abu} where the proof requires coefficient estimation, for odd univalent
function $f(z)=\sum_{k=1}^{\infty} a_{2k-1}z^{2k-1}$, sharp bound for $a_{2k-1}$ is still unknown. Even if we use the known
coefficient bound for odd univalent functions, we do not get better bound for the Bohr radius.


\begin{thm}\label{subtheo}
If $f,g$ are analytic in $\ID$ such that $f(z)=z+\sum_{k=2}^{\infty} a_{2k-1}z^{2k-1}$ is odd univalent in $\ID$ and
$g(z)=\sum_{n=1}^{\infty} b_nz^n\in S(f)$, then $M_g(r)\leq 1$
holds for $r \leq r_*$, where $r_*=0.554958$ is the minimal positive root of the equation
$$
x^2=(1-x)^2(1+x).
$$
\end{thm}

If $g$ in Theorem \ref{subtheo} is also odd analytic, then one can easily obtain the sharp value of the Bohr radius in Theorem \ref{subtheo}
(see Remark \ref{KP2-rem2}). We conclude the section with the following problem.

\bprob \label{KP2-prob2}
Find the Bohr radius for the class of odd functions $f$ satisfying $0<|f(z)|\le 1$ for all $0<|z|<1$.
\eprob

\section{Proofs of Theorems \ref{KP2-th3} and \ref{subtheo}, and Remarks}\label{KP2-sec2}

For the proof of Theorem \ref{KP2-th3}, we need the following lemmas.

\begin{lem}\label{KP2-lem1}
If $r_p$ is the maximal positive root of the equation
$$8r^{2p}+r^{2(p-1)}-6r^{p-1}+1= 0,
$$
then $2{r_p}^{p+1} \leq 1$.
\end{lem}
\begin{proof}  Let $y=r_p^{p+1}$. Then we have a quadratic equation:
$$
(8 +1/r_p^2)y^2-6y+r_p^2=0
$$
which has two solutions
$$
y=\frac{3 \pm 2\sqrt{2}\sqrt{1-r_p^2}}{8 +1/r_p^2} \leq \frac{3 +2\sqrt{2}\sqrt{1-r_p^2}}{8 +1/r_p^2}.
$$
Consequently,
$$2r_p^{p+1}=2y \leq \frac{6 +4\sqrt{2}\sqrt{1-r_p^2}}{8 +1/r_p^2} \leq \sup_{r \in (0,1]}\frac{6 +4\sqrt{2}\sqrt{1-r^2}}{8 +1/r^2}=1,
$$
which completes the proof of Lemma \ref{KP2-lem1}.
\end{proof}

\begin{lem}\label{KP2-lem2}
 Let $|a|<1$ and $0 < R \leq 1$. If $g(z)=\sum_{k=0}^{\infty} b_kz^k$ is analytic and satisfies the inequality $|g(z)| \leq 1$ in $\ID$, then
the following sharp inequality holds:
\begin{equation}\label{KP2-eq3}
\sum_{k=1}^\infty |b_k|^2R^{pk} \leq R^{p}\frac{(1-|b_0|^2)^2}{1-|b_0|^2R^{p}}.
\end{equation}
\end{lem}
\begin{proof}
Let $b_0=a$. Then, it is easy to see that the condition on $g$ can be rewritten in terms of subordination as
\begin{equation}\label{KP2-eq1}
g(z) = \sum_{k=0}^{\infty} b_kz^k \prec \phi (z),
\end{equation}
where
$$\phi (z)=\frac{a-z}{1-\overline{a}z}=a-(1-|a|^2)\sum_{k=1}^{\infty}(\overline{a})^{k-1} z^{k}, \quad z\in\ID.
$$
Note that $\phi$ is analytic in $\ID$ and $|\phi (z)|\leq 1$ for $z\in\ID$. The subordination relation \eqref{KP2-eq1} gives
$$\sum_{k=1}^\infty |b_k|^2R^{2k} \leq (1-|a|^2)^2\sum_{k=1}^{\infty}|a|^{2(k-1)} R^{2k}
 =R^2 \frac{(1-|a|^2)^2}{1-|a|^2R^2}
$$
from which we arrive at the inequality \eqref{KP2-eq3} which proves Lemma \ref{KP2-lem2}.
\end{proof}

\begin{proof}[Proof of Theorem \ref{KP2-th3}]
Let $r=r_p$ and $f(z)=\sum_{k=0}^{\infty} a_{pk+1} z^{pk+1}$, where $|f(z)|\le 1$ for $z\in\ID$.
At first, we remark that the function $f$ can be represented as $f(z)=z g(z^p)$, where $|g(z)| \leq 1$ in $\ID$ and
$g(z)=\sum_{k=0}^{\infty} b_kz^k$ is analytic in $\ID$  with $b_k=a_{pk+1}$.  Let $|b_0|=a$. Choose any $\rho >1$ such that $\rho r \leq 1$. Then it follows that
\begin{eqnarray*}
\sum_{k=1}^{\infty} |a_{pk+1}|r^{pk} &=& \sum_{k=1}^{\infty}|b_k| r^{pk}  \\
&\leq & \sqrt{\sum_{k=1}^{\infty}|b_k|^2 \rho^{pk}r^{pk}} \sqrt{\sum_{k=1}^{\infty} \rho^{-pk}r^{pk}} \\
&\leq & \sqrt{r^p\rho^p\frac{(1-a^2)^2}{1-a^2r^p\rho^p}}\, \sqrt{\frac{\rho^{-p}r^p}{1-\rho^{-p}r^{p}}}\\
&= &  \frac{r^p(1-a^2)}{\sqrt{1-a^2r^p\rho^p}}\,  \frac{1}{\sqrt{1-\rho^{-p}r^{p}}}.
\end{eqnarray*}

In the second and the third steps above we have used the classical Cauchy-Schwarz inequality and \eqref{KP2-eq3} with $R=\rho r$, respectively.
 Hence
\begin{equation} \label{Additional}
\sum_{k=1}^{\infty} |a_{pk+1}|r^{pk} \leq \frac{r^p(1-a^2)}{\sqrt{1-a^2r^p\rho^p}}\,  \frac{1}{\sqrt{1-\rho^{-p}r^{p}}}.
\end{equation}

We need to consider the cases $a \ge r^p$ and $a<r^p$ separately.

\noindent
\vspace{6pt}
{\bf Case 1: $a \ge r^p$.} In this case set $\rho=1/\sqrt[p]{a}$ and
obtain
\begin{equation}\label{KP2-eq4}
\sum_{k=0}^{\infty}|a_{pk+1}|r^{pk+1} \leq r \left(a+r^p\frac{(1-a^2)}{1-r^p a}\right).
\end{equation}
For convenience, we may let $\alpha =r^p$ and consider
$$\psi (x)=x+\alpha  \frac{(1-x^2)}{1- \alpha x}, \quad x\in [0,1].
$$
Finally, we just need to maximize $\psi (x)$ over the interval $[0,1]$. We see that $\psi$ has two critical points and obtain
that the maximum occurs when
$$x_1=\left (1-\frac{\sqrt{1-\alpha ^2}}{\sqrt{2}}\right )\frac{1}{\alpha }, \quad \alpha \ge \frac{1}{3},
$$
and thus, $\psi (x)\leq \psi (x_1)$. Consequently, by \eqref{KP2-eq4}, we find that
\begin{equation} \label{Finish1}
\sum_{k=0}^{\infty}|a_{pk+1}|r^{pk+1} \leq \frac{1}{r^{p-1}}\left(3-2 \sqrt{2}\sqrt{1-r^{2p}} \right)=1.
\end{equation}

\noindent
\vspace{6pt}
{\bf Case 2: $a < r^p$.} In this case we set $\rho=1/r$ and apply (\ref{Additional}). As a result we get
\begin{equation} \label{Finish2}
\sum_{k=0}^{\infty} |a_{pk+1}|r^{pk+1} \leq r(a+r^p \sqrt{1-a^2}/\sqrt{1-r^{2p}}) \leq 2r^{p+1} \leq 1.
\end{equation} Here we omitted the critical point $a=\sqrt{1-r^{2p}}$ because it is less than or equal to $r^p$ only in the case $r^{2p}>1/2$ which
contradicts Lemma \ref{KP2-lem1}.

The last inequality in \eqref{Finish2} follows from Lemma \ref{KP2-lem1}.

Now,  (\ref{Finish1}) and (\ref{Finish2}) finish the proof of the first part of Theorem \ref{KP2-th3}. Now we have to say a few words about extremal.
We set $f(z)=z(z^p-a)/(1-az^p)$ with
$a=\left (1-\frac{\sqrt{1-{r}^{2p}}}{\sqrt{2}}\right )/r^p$ and then calculate the Bohr radius for it. It coincides with $r$.

Certainly, an extremal function is unique up to a rotation of $a$. To see this we just trace our inequalities and see that the equality holds only when $|b_0|=a$.
\end{proof}

\begin{proof}[Proof of Theorem \ref{subtheo}]
Let $g\prec f$, where $g(z)=\sum_{n=1}^{\infty} b_nz^n$, and $f(z)=\sum_{k=1}^{\infty} a_{2k-1}z^{2k-1}$ is an odd univalent in $\ID$. Here $a_1=1$
and thus, by the definition of the subordination, $|b_1|\leq 1$. First we show that
\begin{equation}\label{KP2-eq6}
\sum_{k=1}^\infty |a_{2k-1}| r^{2k-1} \leq \frac{r}{1-r^2}\quad \mbox{ for $|z|=r<1$}.
\end{equation}
In order prove this, we use Robertson's inequality for odd univalent function $f$ (see for instance, \cite[Section 2.2]{AvWir-09}),
$$\sum_{k=1}^n |a_{2k-1}|^2  \leq n.
$$
Using this, we derive that
\begin{equation}\label{KP2-eq6a}
S_n = \sum_{k=1}^n |a_{2k-1}|  \leq \sqrt{n} \sqrt{\sum_{k=1}^n |a_{2k-1}|^2} \leq n.
\end{equation}
It follows from \eqref{KP2-eq6a} that
$$\sum_{k=1}^\infty |a_{2k-1}| r^{2k-1}=|a_1|(r-r^3)+\sum_{k=2}^\infty S_k(r^{2k-1}-r^{2k+1})
\leq r-r^3+\sum_{k=2}^\infty k(r^{2k-1}-r^{2k+1})=\frac{r}{1-r^2}
$$
which proves \eqref{KP2-eq6}.

Next, as $g\prec f$, we have by \eqref{KP2-eq6} that
$$\sum_{k=1}^\infty |b_k|^2r^{2k} \leq \sum_{k=1}^\infty |a_{2k-1}|^2r^{2(2k-1)} \leq \frac{r^2}{1-r^4}
$$
which gives
\begin{equation}\label{KP2-eq7}
\sum_{k=1}^\infty |b_k|^2r^{k}  \leq \frac{r}{1-r^2}.
\end{equation}
Consequently, by the classical Cauchy-Schwarz inequality, we obtain
$$\sum_{k=1}^\infty |b_k|r^{k}  \leq
\sqrt{\sum_{k=1}^\infty |b_k|^2r^{k}}\sqrt{\sum_{k=1}^\infty r^{k}} \leq \sqrt{\frac{r}{1-r^2}\frac{r}{1-r}}=\frac{r}{(1-r)\sqrt{1+r}}
$$
which is less than or equal to $1$ if $r^3-2r^2-r+1\geq 0$.
\end{proof}

\begin{remark}\label{1KP2-rem1}
{\rm
Now we show a way that can slightly improve the Bohr radius in Theorem \ref{subtheo}.
We see that
\begin{eqnarray*}
\sum_{k=1}^\infty |b_k|^2r^{k} &=& |b_1|r+|b_2|r^2+\sum_{k=3}^\infty |b_k|r^{k}\\
& \leq & |b_1|r+|b_2|r^2+\sqrt{\sum_{k=3}^\infty |b_k|^2r^{k}}\sqrt{\sum_{k=3}^\infty r^{k}}\\
& \leq & \psi (|b_1|,|b_2|),
\end{eqnarray*}
where we have used \eqref{KP2-eq7} and the function
\be\label{KP2-eq8}
\psi (x,y) =rx+r^2y+\frac{r^2}{\sqrt{1-r}}\sqrt{\frac{1}{1-r^2}-x^2-ry^2}
\ee
with $x=|b_1|$ and $y=|b_2|$. Therefore, we have to find
$$\max \{\psi (x,y):\, 0\leq x\leq 1, \,x^2+y \leq 1 \}.
$$

At first consider the case $y<1-x^2$. In this case
$$\frac{\partial}{\partial y}\psi (x,y)=0 \mbox{ or } y=0.
$$ 
If $y=0$, then $\max \{\psi (x,0):\, 0\leq x\leq 1\}=1121 \sqrt{7/53}/410 < 1$ for $r=0.59$ which is too big.
Now consider the case
$$
\frac{\partial}{\partial y}\psi (x,y)=r^2 - \frac{\sqrt{1/(1 - r)} r^3 y}{\sqrt{1/(1 - r^2) - x^2 - r y^2}}=0
$$ so that
$$
y=\frac{\sqrt{1 - x^2 + r^2 x^2}}{\sqrt{r + r^2}}>1-x^2 \mbox{ for } r \leq 0.6.
$$
Therefore, we may assume that $y=1-x^2$. In this case we have to verify that
 $$\max \{\psi (x,1-x^2) :\, 0\leq x \leq 1 \} \leq 1,
$$
where $\psi (x,1-x^2)$ is obtained from \eqref{KP2-eq8} by letting $y=1-x^2$.
Straightforward and routine computations show that $r=0.564 \ldots$ which is slightly better than the estimate presented in Theorem \ref{subtheo}.

We conclude that the Bohr radius in this case cannot be greater than $(\sqrt{5}-1)/2=0.618034 \ldots$. This upper bound can be easily
obtained from the example $f(z)=z/(1-z^2)$.
}
\end{remark}

\begin{remark}\label{KP2-rem2}
{\rm If $g\prec f$, where $g(z)=\sum_{k=1}^{\infty} b_{2k-1}z^{2k-1}$ is also odd, and $f(z)=\sum_{k=1}^{\infty} a_{2k-1}z^{2k-1}$
is an odd univalent in $\ID$ (with $a_1=1)$, then by Rogosinski's theorem \cite{Rogo43} and then  Robertson's inequality,
we obtain that
$$\sum_{k=1}^n |b_{2k-1}|^2\leq \sum_{k=1}^n |a_{2k-1}|^2 \leq n
$$
which, as in the proof of Theorem \ref{subtheo}, implies that
$$ \sum_{k=1}^\infty  |b_{2k-1}|r^{2k-1} \leq \frac{r}{1-r^2}.
$$
Note that $ \frac{r}{1-r^2} = 1$ gives $r=(\sqrt{5}-1)/2$ and thus, in this case, we have the sharp Bohr radius and the extremal function is $f(z)=z/(1-z^2)$.
}
\end{remark}

While determining the Bohr radius in the case of  functions $f$ analytic in the unit disk, one often requires sharp estimate on the
Taylor coefficients of  $f$. In the class of \textit{odd} univalent functions, sharp coefficient estimate is still unknown unlike for the class
of all univalent analytic functions solved by de Branges. In spite of this drawback, it is interesting to state the
following sharp result as a corollary to the relation \eqref{KP2-eq6}.

\begin{cor}\label{KP2-cor5}
If $f(z)=\sum_{k=0}^{\infty} a_{2k+1}z^{2k+1}$ is analytic in $\ID$ and univalent in $\ID$, where $0<\alpha =|a_1|\leq 1$, then
$$\sum_{k=0}^{\infty} |a_{2k+1}|r^{2k+1} \leq 1 \mbox{ for } r \leq r_{\alpha}=\frac{-\alpha +\sqrt{4+\alpha ^2}}{2}.
$$
The extremal function has the form $\alpha z/(1-z^2)$.
\end{cor}
\begin{proof}
By hypotheses, the relation \eqref{KP2-eq6} implies that for $|z|=r<1$,
$$\sum_{k=0}^\infty |a_{2k+1}| r^{2k+1} \leq \frac{\alpha r}{1-r^2}
$$
which is less than or equal to $1$ for $r\leq r_{\alpha}$. Observe that in the normalized case (i.e. $a_1 =1$), the radius $r_{\alpha}$ gives the value $(\sqrt{5}-1)/2\, =0.618034 \ldots$.
\end{proof}

\subsection*{Acknowledgements}
The research of the first author was supported by Russian foundation for basic research, Proj. 17-01-00282, and the research of the second author was supported
by the project RUS/RFBR/P-163 under Department of Science \& Technology (India).
The second author is currently on leave from the IIT Madras.

\end{document}